\documentclass[12pt]{article}
\usepackage{amssymb}
\usepackage{amsthm}
\usepackage{amsmath}
\usepackage{graphicx}
\newtheorem{theorem}{Theorem}
\newtheorem{definition}[theorem]{Definition}

\newtheorem{remark}[theorem]{Remark}
\newtheorem{example}[theorem]{Example}

\title{Orthomodular lattices can be converted into left residuated l-groupoids\thanks{Preprint of an article published by Miskolc University Press in Miskolc Mathematical Notes 18 (2017), No.\ 2, pp.\ 685--689. DOI: 10.18514/MMN.2017.1730. It is available online at: \texttt{mat76.mat.uni-miskolc.hu/mnotes}.}}
\author{Ivan Chajda and Helmut L\"anger}
\date{}
\begin{document}
\footnotetext[1]{Support of the research of both authors by the bilateral project entitled ``New perspectives on residuated posets'', supported by the Austrian Science Fund (FWF), project I~1923-N25, and the Czech Science Foundation (GA\v CR), project 15-34697L, as well as by \"OAD, project 04/2017, is gratefully acknowledged.}
\maketitle
\begin{abstract}
We show that every orthomodular lattice can be considered as a left residuated l-groupoid satisfying divisibility, antitony, the double negation law and three more additional conditions expressed in the language of residuated structures. Also conversely, every left residuated l-groupoid satisfying the mentioned conditions can be organized into an orthomodular lattice.
\end{abstract}
 


It is well-known that residuated structures form an algebraic axiomatization of fuzzy logics, see e.~g.\ \cite{Bel} for an overview. The reader can find necessary concepts and definitions concerning residuated structures in \cite{GJKO}, however this paper is self-contained. Orthomodular lattices were introduced by G.~Birkhoff and J.~von~Neumann as an algebraic axiomatization of the logic of quantum mechanics, see e.~g.\ \cite{DP}, \cite{K} or \cite{Ber} for details. Hence it is a natural question if these two concepts have a common base, i.~e.\ if orthomodular lattices can be considered as residuated structures and hence as an axiomatization of certain fuzzy logic and, conversely, if certain residuated structures can be converted into orthomodular lattices, i.~e.\ if the logic of quantum mechanics can be considered as a kind of fuzzy logic. For the theory of orthomodular lattices cf.\ the monographs \cite{K} and \cite{Ber} as well as the paper \cite{CHK}.

We start with the definition of an orthomodular lattice.

\begin{definition}\label{def1}
An {\em orthomodular lattice} is an algebra ${\mathcal L}=(L,\vee,\wedge,',0,1)$ of type $(2,2,1,$ $0,0)$ satisfying {\rm(i)} -- {\rm(v)} for all $x,y\in L$:
\begin{enumerate}
\item[{\rm(i)}] $(L,\vee,\wedge,0,1)$ is a bounded lattice.
\item[{\rm(ii)}] $x\vee x'=1$
\item[{\rm(iii)}] $x\leq y$ implies $y'\leq x'$.
\item[{\rm(iv)}] $(x')'=x$
\item[{\rm(v)}] $x\leq y$ implies $y=x\vee(y\wedge x')$.
\end{enumerate}
\end{definition}

\begin{remark}\label{rem2}
In every lattice $(L,\vee,\wedge)$ with a unary operation $'$ satisfying {\rm(iii)} and {\rm(iv)} the so-called de Morgan laws
\[
(x\vee y)'=x'\wedge y'\mbox{ and }(x\wedge y)'=x'\vee y'
\]
hold.
\end{remark}

\begin{remark}\label{rem3}
According to the de Morgan laws condition {\rm(v)} can be replaced by
\begin{enumerate}
\item[{\rm(vi)}] $x\leq y$ implies $x=y\wedge(x\vee y')$.
\end{enumerate}
\end{remark} 

Now we introduce left residuated l-groupoids.

\begin{definition}\label{def2}
A {\em left residuated l-groupoid} is an algebra ${\mathcal A}=(A,\vee,\wedge,\odot,\rightarrow,0,1)$ of type $(2,2,2,2,0,0)$ satisfying {\rm(i)} -- {\rm(iii)} for all $x,y,z\in A$:
\begin{enumerate}
\item[{\rm(i)}] $(A,\vee,\wedge,0,1)$ is a bounded lattice.
\item[{\rm(ii)}] $x\odot1=1\odot x=x$.
\item[{\rm(iii)}] $x\odot y\leq z$ if and only if $x\leq y\rightarrow z$.
\end{enumerate}
Condition {\rm(iii)} is called {\em left adjointness}. ${\mathcal A}$ is said to satisfy {\em divisibility} if
\[
(x\rightarrow y)\odot x=x\wedge y
\]
for all $x,y\in A$. We define a unary operation $'$ on $A$ by
\[
x':=x\rightarrow0
\]
for all $x\in A$. ${\mathcal A}$ is said to satisfy {\em antitony} if
\[
x\leq y\mbox{ implies }y'\leq x'
\]
for all $x,y\in A$ and ${\mathcal A}$ is said to satisfy the {\em double negation law} if
\[
(x')'=x
\]
for all $x\in A$.
\end{definition}

\begin{example}\label{ex1}
If $A:=\{0,a,a',b,b',1\}$, $(A,\vee,\wedge,0,1)$ denotes the bounded lattice with the Hasse diagram
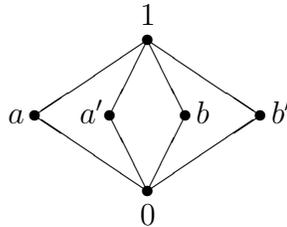
\begin{center}
\setlength{\unitlength}{5mm}
\begin{picture}(6,5)
\put(4,1){\circle*{.3}}
\put(1,3){\circle*{.3}}
\put(3,3){\circle*{.3}}
\put(5,3){\circle*{.3}}
\put(7,3){\circle*{.3}}
\put(4,5){\circle*{.3}}
\put(4,1){\line(-3,2){3}}
\put(4,1){\line(-1,2){1}}
\put(4,1){\line(1,2){1}}
\put(4,1){\line(3,2){3}}
\put(4,5){\line(-3,-2){3}}
\put(4,5){\line(-1,-2){1}}
\put(4,5){\line(1,-2){1}}
\put(4,5){\line(3,-2){3}}
\put(.3,2.8){$a$}
\put(2.2,2.8){$a'$}
\put(5.3,2.8){$b$}
\put(7.3,2.8){$b'$}
\put(3.8,.1){$0$}
\put(3.8,5.4){$1$}
\end{picture}
\end{center}
and the binary operations $\odot$ and $\rightarrow$ are defined by the tables
\[
\begin{array}{c|cccccc}
\odot & 0 & a & a' & b & b' & 1 \\
\hline
0 & 0 & 0 & 0 & 0 & 0 & 0 \\
a & 0 & a & 0 & b & b' & a \\
a' & 0 & 0 & a' & b & b' & a' \\
b & 0 & a & a' & b & 0 & b \\
b' & 0 & a & a' & 0 & b' & b' \\
1 & 0 & a & a' & b & b' & 1
\end{array}
\quad\quad
\begin{array}{c|cccccc}
\rightarrow & 0 & a & a' & b & b' & 1 \\
\hline
0 & 1 & 1 & 1 & 1 & 1 & 1 \\
a & a' & 1 & a' & a' & a' & 1 \\
a' & a & a & 1 & a & a & 1 \\
b & b' & b' & b' & 1 & b' & 1 \\
b' & b & b & b & b & 1 & 1 \\
1 & 0 & a & a' & b & b' & 1
\end{array}
\] 
then $(A,\vee,\wedge,\odot,\rightarrow,0,1)$ is a left residuated l-groupoid satisfying divisibility, antitony and the double negation law. The mentioned lattice is the smallest orthomodular lattice which is not a Boolean algebra and it is usually denoted by {\rm MO2}.
\end{example}

The following theorem says that to every orthomodular lattice there can be assigned a left residuated l-groupoid in a natural way.

\begin{theorem}\label{th1}
Let ${\mathcal L}=(L,\vee,\wedge,',0,1)$ be an orthomodular lattice and define binary operations $\odot$ and $\rightarrow$ on $L$ by the following formulas:
\begin{equation}\label{equ1}
x\odot y=(x\vee y')\wedge y,
\end{equation}
\begin{equation}\label{equ2}
x\rightarrow y=(y\wedge x)\vee x'.
\end{equation}
Then ${\bf A}({\mathcal L})=(L,\vee,\wedge,\odot,\rightarrow,0,1)$ is a left residuated l-groupoid satisfying divisibility, antitony, the double negation law as well as the following identity:
\begin{equation}\label{equ3}
x\odot(x\vee y)=x.
\end{equation}
Moreover, $x'=x\rightarrow0$ for all $x\in L$.
\end{theorem}

\begin{proof}
Let $a,b\in L$. We have
\[
a\rightarrow0=(0\wedge a)\vee a'=0\vee a'=a'.
\]
Of course, $(L,\vee,\wedge,0,1)$ is a bounded lattice. Moreover,
\begin{eqnarray*}
& & a\odot1=(a\vee1')\wedge1=(a\vee0)\wedge1=a\wedge1=a\mbox{ and} \\
& & 1\odot a=(1\vee a')\wedge a=1\wedge a=a.
\end{eqnarray*}
If $a\odot b\leq c$ then $(a\vee b')\wedge b\leq c$ and hence
\[
a\leq a\vee b'=((a\vee b')\wedge b)\vee b'=(((a\vee b')\wedge b)\wedge b)\vee b'\leq(c\wedge b)\vee b'=b\rightarrow c.
\]
If, conversely, $a\leq b\rightarrow c$ then $a\leq(c\wedge b)\vee b'$ and hence
\[
a\odot b=(a\vee b')\wedge b\leq(((c\wedge b)\vee b')\vee b')\wedge b=((c\wedge b)\vee b')\wedge b=c\wedge b\leq c.
\]
Now, using orthomodularity (i.~e.\ (v) of Definition~\ref{def1}), we have
\[
(a\rightarrow b)\odot a=(((b\wedge a)\vee a')\vee a')\wedge a=((b\wedge a)\vee a')\wedge a=a\wedge b.
\]
In view of Definition~\ref{def1}, $a\leq b$ implies $b'\leq a'$ and we have $(a')'=a$. Finally, by applying (\ref{equ1}) and (vi) of Definition~\ref{def1} we obtain
\[
a\odot(a\vee b)=(a\vee(a\vee b)')\wedge(a\vee b)=a.
\]
\end{proof}

\begin{remark}\label{rem1}
The operation $x\odot y:=(x\vee y')\wedge y$ is called the {\em Sasaki projection} of $x$ onto $y$ {\rm(}cf.\ {\rm\cite{K}} and {\rm\cite{Ber})}.
\end{remark}

Conversely, certain left residuated l-groupoids give rise to an orthomodular lattice.

\begin{theorem}\label{th2}
Let ${\mathcal A}=(A,\vee,\wedge,\odot,\rightarrow,0,1)$ be a left residuated l-groupoid satisfying antitony, the double negation law as well as identities {\rm(\ref{equ1})} and {\rm(\ref{equ3})} of Theorem~\ref{th1}. Moreover, define $x':=x\rightarrow0$ for all $x\in A$. Then ${\bf L}({\mathcal A})=(A,\vee,\wedge,',0,1)$ is an orthomodular lattice.
\end{theorem}

\begin{proof}
Let $a,b\in A$. Clearly, $(A,\vee,\wedge,0,1)$ is a bounded lattice and $a'=a\rightarrow0$. Using antitony we see that $a\leq b$ implies $b'\leq a'$. Moreover, we have $(a')'=a$ according to the double negation law. Finally, if $a\leq b$ then, using (\ref{equ3}) and (\ref{equ1}), we have
\[
b=(b')'=(b'\odot(b'\vee a'))'=(b'\odot a')'=((b'\vee a)\wedge a')'=a\vee(b\wedge a')
\]
and hence $a\vee a'=a\vee(1\wedge a')=1$.
\end{proof}

Finally, we prove that the correspondence described in the last two theorems is one-to-one.

\begin{theorem}\label{th3}
We have ${\bf L}({\bf A}({\mathcal L}))={\mathcal L}$ for every orthomodular lattice ${\mathcal L}$ and ${\bf A}({\bf L}({\mathcal A}))={\mathcal A}$ for every left residuated l-groupoid satisfying antitony, the double negation law as well as identities {\rm(\ref{equ1})} -- {\rm(\ref{equ3})} of Theorem~\ref{th1}.
\end{theorem}

\begin{proof}
If ${\mathcal L}=(L,\vee,\wedge,',0,1)$ is an orthomodular lattice, ${\bf A}({\mathcal L})=(L,\vee,\wedge,\odot,\rightarrow,0,1)$ and ${\bf L}({\bf A}({\mathcal L}))=(L,\vee,\wedge,^*,0,1)$ then
\[
x^*=x\rightarrow0=(0\wedge x)\vee x'=0\vee x'=x'
\]
for all $x\in L$, therefore we obtain ${\bf L}({\bf A}({\mathcal L}))={\mathcal L}$. Conversely, if ${\mathcal A}=(A,\vee,\wedge,\odot,$ $\rightarrow,0,1)$ is a left residuated l-groupoid satisfying divisibility, antitony, the double negation law as well as identities (\ref{equ1}) -- (\ref{equ3}) of Theorem~\ref{th1}, ${\bf L}({\mathcal A})=(A,\vee,\wedge,',0,1)$ and ${\bf A}({\bf L}({\mathcal A}))=(A,\vee,\wedge,\circ,\Rightarrow,0,1)$ then
\begin{eqnarray*}
x\circ y & = & (x\vee y')\wedge y=x\odot y\mbox{ and} \\
x\Rightarrow y & = & (y\wedge x)\vee x'=x\rightarrow y
\end{eqnarray*}
for all $x,y\in A$, therefore we obtain ${\bf A}({\bf L}({\mathcal A}))={\mathcal A}$.
\end{proof}

\begin{remark}\label{rem4}
We have shown that orthomodular lattices can be considered as special residuated lattices and hence the logic of quantum mechanics axiomatized by them has a common base with a certain fuzzy logic axiomatized just by means of residuated lattices as pointed out in {\rm\cite{Bel}}. This sheds a new light on the logic of quantum mechanics and yields new tools for its investigation.
\end{remark}

Authors' addresses:

Ivan Chajda \\
Palack\'y University Olomouc \\
Faculty of Science \\
Department of Algebra and Geometry \\
17.\ listopadu 12 \\
771 46 Olomouc \\
Czech Republic \\
ivan.chajda@upol.cz

Helmut L\"anger \\
 TU Wien \\
Faculty of Mathematics and Geoinformation \\
Institute of Discrete Mathematics and Geometry \\
Wiedner Hauptstra\ss e 8-10 \\
1040 Vienna \\
Austria \\
helmut.laenger@tuwien.ac.at
\end{document}